\documentclass{amsart}
\usepackage{amsfonts}
\usepackage{graphicx}
\usepackage{amscd}
\usepackage{amsmath}
\usepackage{amssymb}

\makeatletter
\@namedef{subjclassname@2010}{%
  \textup{2010} Mathematics Subject Classification}
\makeatother

\theoremstyle{plain}

\newtheorem{corollary}{\bf Corollary}

\newtheorem{lemma}{\bf Lemma}

\newtheorem{remark}{Remark}

\newtheorem{theorem}{\bf Theorem}

\theoremstyle{definition}

\numberwithin{equation}{section}

\title[quasi Einsten manifold]{On the classification of noncompact steady quasi-Einstein manifold with vanishing condition on the Weyl tensor}

\author{H. Baltazar}
\author{M. Matos Neto}

\address[H. Baltazar]{Departamento de Matem\'{a}tica, Universidade Federal do Piau\'{\i}\\
64049-550 Te\-re\-si\-na, Piau\'{\i}, Brazil.}
\email{halyson@ufpi.edu.br}

\address[M. Matos Neto]{Departamento de Matem\'{a}tica, Universidade Federal do Piau\'{\i}\\
64049-550 Te\-re\-si\-na, Piau\'{\i}, Brazil.}
\email{mvieira@ufpi.edu.br}

\subjclass[2010]{Primary 53C25, 53C20, 53C21; Secondary 53C65}
\keywords{Quasi-Einstein manifold; warped product; Weyl tensor}
\date{July 28, 2017}

\begin{document}

\newcommand{\spacing}[1]{\renewcommand{\baselinestretch}{#1}\large\normalsize}
\spacing{1.2}

\begin{abstract}
The aim of this paper is to study complete (noncompact) steady $m$-quasi-Einstein manifolds satisfying a fourth-order vanishing condition on the Weyl tensor. In this case, we are able to prove that a steady $m$-quasi-Einstein manifold ($m>1$) on a simply connected $n$-dimensional manifold $(M^{n},g)$, $(n\geq4),$ with nonnegative Ricci curvature and zero radial Weyl curvature must be a warped product  with $(n-1)-$dimensional Einstein fiber, provided that $M$ has fourth order divergence-free Weyl tensor (i.e., ${\rm div}^{4}W=0$).
\end{abstract}

\maketitle

\section{Introduction}
\label{intro}

Following the terminology used in \cite{CC,Ribeiro} we recall the definition of quasi-Einstein manifold. A complete Riemannian manifold $(M^{n},g)$, $n\geq2,$ will be called $m$-quasi-Einstein manifold or simply quasi-Einstein, if there exist a smooth potential function f on $M,$ a constant $m$ with $0<m\leq+\infty$  and a constant $\lambda$ satisfying
\begin{equation}\label{eqqEinstein}
Ric+Hessf-\frac{1}{m}df\otimes df=\lambda g,
\end{equation}
where $Hessf$ stands for a Hessian of $f$. We shall refer to this equation as the fundamental equation of a quasi-Einstein manifold $(M^{n},g,f,\lambda)$.
We also recall the well known $m$-Bakry-Emery Ricci tensor that is given by
$$Ric_{f}^{m}=Ric+Hessf-\frac{1}{m}df\otimes df,$$
where $f$ is a smooth function on $M^{n}$,  for more details about this tensor see for instance \cite{Bakry,Qian,WW}. Therefore, the fundamental equation of the quasi-Einstein (\ref{eqqEinstein}) can be rewritten as
\begin{equation}\label{eqricf}
Ric_{f}^{m}=\lambda g.
\end{equation}
A quasi-Einstein manifold is called shrinking if $\lambda>0,$ steady if $\lambda=0$ and expanding if $\lambda<0.$

\begin{remark}
It is important highlight that, if $(M,g)$ and $(F,g_{F})$ are Riemannian manifolds, the warped product $(M\times F,\overline{g}=g+e^{-\frac{2}{m}f}g_{F})$ is an Einstein manifold with Einstein constant $\lambda$ if and only if $(F,g_{F})$ is Einstein ($Ric_{g_{F}}=\mu g_{F}$), satisfies (\ref{eqricf}) for metric $g,$ and
\begin{equation}\label{mu}
\Delta f=|\nabla f|^{2}+m\lambda-m\mu e^{\frac{2}{m}f}.
\end{equation}
In 2003, Kim and Kim showed that every quasi-Einstein manifold must satisfy (\ref{mu}) for some constant $\mu.$ For more references on Einstein warped products and quasi-Einstein metrics, see \cite{batista15,Silva14,besse,CSW,KK,Rimoldi}.
\end{remark}

Let us point out that if $m=\infty$ the equation (\ref{eqqEinstein}) becomes the fundamental equation of the gradient Ricci solitons, in this case, we refer the reader to the survey papers \cite{Cao10,Cao11} and references therein for a overview on this subject. Moreover, recent results can be found in
\cite{CC12,CaoChen,CCCMM,catino13,cmm,FG11} and \cite{MS13}.

In 2012, Chenxu He, Petersen and Wylie proved that, if a complete, simply connected manifold has harmonic Weyl tensor and satisfies $W(\cdot,\nabla f,\cdot,\nabla f)=0$, then $(M,g)$ is a $m$-quasi Einstein metric if and only if it is a warped product with Einstein fibers. More precisely, they proved the following result.

\begin{theorem}[He-Petersen-Wylie, \cite{CPW}]\label{cpwylie}
Let $m>1$ and suppose that $(M,g)$ is complete, simply connected, and has harmonic Weyl tensor and $W(\nabla f,\cdot,\cdot,\nabla f)=0$, then $(M,g,f)$ is a nontrivial $m$-quasi-Einstein metric if and only if it is of the form
$$g=dt^{2}+\psi^{2}(t)g_{L}\;\;\;and\;\;\;f=f(t),$$
where $g_{L}$ is an Einstein metric. Moreover, if $\lambda\geq0$ then $(L,g_{L})$ has non-negative Ricci curvature, and if it is Ricci flat, then $\psi$ is constant, i.e., $(M^{n},g)$ is a Riemannian product.
\end{theorem}

Later on, assuming that the manifold is Bach-flat, Chen and He showed in \cite{CC} that any shrinking quasi-Einstein manifold is either Einstein or a finite quotient of a warped product with $(n-1)$-dimensional Einstein fiber. At this point, it is important to say that if $m$ is positive then a quasi-Einstein manifold is compact if and only if $\lambda>0$ (see, for example, \cite{KK,Qian} and \cite{WW07} for a complete description). More recently, Ranieri and Ribeiro Jr. \cite{Ribeiro} have studied steady quasi-Einstein metrics under Bach-flat assumption. In this case, the authors proved that a Bach-flat noncompact steady quasi-Einstein manifold with positive Ricci curvature must be a warped product with Einstein fiber. In this paper, motivated by the historical development on the study of the quasi-Einstein manifolds, we shall investigate such structure satisfying a fourth-order vanishing condition on the Weyl tensor.

Before presenting our first result, it is fundamental to remember that a Riemannian manifold $(M^{n}, g)$ has {\it zero radial Weyl curvature} when
$$W(\cdot,\cdot,\cdot,\nabla f)=0,$$
for a suitable smooth function $f.$ This condition have been used to classify quasi-Einstein manifolds or more general generalized quasi-Einstein manifolds, see for instance, \cite{catino,CPW,Bleandro} and \cite{pwylie}.

In the sequel, in the the same spirit of the recent work due Catino, Mastrolia and Monticelli \cite{cmm}, let us introduce the following definitions:
$${\rm div}^{4}W=\nabla_{k}\nabla_{j}\nabla_{i}\nabla_{l}W_{ijkl}$$
and
$${\rm div}^{3}C=\nabla_{k}\nabla_{j}\nabla_{i}C_{ijk},$$
where $W$ and $C$ are the Weyl and the Cotton tensors, respectively (see the definitions of this tensors in the Section \ref{Preliminaries}). In \cite{cmm}, the authors showed that n-dimensional complete gradient shrinking Ricci solitons with fourth order divergence free Weyl tensor (i.e., ${\rm div}^{4}W=0$) are either Einstein or finite quotients of  $N^{n-k}\times \mathbb{R}^{k},$ $(k>0).$ That is, the product of a Einstein manifold $N^{n-k}$ with the Gaussian shrinking soliton $\mathbb{R}^{k}.$ Recently, other rigidity results have been obtained under vanishing condition on the Weyl tensor, see, for example,  \cite{YZ,YWZ}.

With this notation in mind, Ranieri and Ribeiro Jr. showed in \cite{Ribeiro} that, under certain appropriate constraints, a Bach-flat noncompact steady quasi-Einstein manifold must have fourth order divergence-free Weyl tensor and satisfies the zero radial Weyl curvature condition (i.e., $W_{ijkl}\nabla_{l}f=0$). It is natural to ask if the converse of this statement is also true. In this sense, inspired by ideas outlined in \cite{CCCMM} and \cite{Ribeiro}, we shall replace the assumption of Bach-flat in \cite[Theorem 1]{Ribeiro} by the condition that $M$ has fourth order divergence-free Weyl tensor. In fact, we have the following result.

\begin{theorem}\label{thmA}
Let $(M^{n},g,f,m>1),$ $n\geq4,$ be a noncompact steady quasi-Einstein manifold with positive Ricci curvature such that $f$ has at least one critical point. If, in addition, $(M,g)$ has zero Weyl radial curvature and satisfies ${\rm div}^{4}W=0,$ then $M^{n}$ has harmonic Weyl tensor.
\end{theorem}

It is well known that $4$-dimensional compact Riemannian manifolds have special behavior. In this case, we are able to conclude that when we restrict the Theorem~\ref{thmA} to the 4-dimensional case, then $M^{4}$ is actually a locally conformally flat manifold. Let us highlight that in \cite{CPW} the authors provided some examples of quasi-Einstein which have ${\rm div}^{4}W=0$ and zero radial Weyl curvature but are not locally conformally flat, cf. \cite[Section 3, Table 2]{CPW}. However, these examples exist only for $n\geq5.$ Hence, after these considerations, we shall apply the Theorem~\ref{thmA} in order to get the following result.

\begin{corollary}\label{corA}
Let $(M^{4},g,f,m>1)$ be a $4$-dimensional noncompact steady quasi-Einstein manifold with positive Ricci curvature such that $f$ has at least one critical point. If in addition $(M,g)$ has zero Weyl radial curvature and satisfies ${\rm div}^{4}W=0$ then it is a locally conformally flat manifold.
\end{corollary}

Finally, as an immediate consequence of Theorems~\ref{cpwylie} and \ref{thmA} we get the following classification result for steady quasi-Einstein manifold.

\begin{theorem}\label{thmB}
Let $(M^{n},g,f,\lambda,m>1),$ $n\geq4,$ be a noncompact, simply connected steady quasi-Einstein manifold with positive Ricci curvature such that $f$ has at least one critical point. If in addition $(M,g)$ has zero Weyl radial curvature and satisfies ${\rm div}^{4}W=0,$ then $(M^{n},g)$ is a warped product with
$$g=dt^{2}+\psi^{2}(t)g_{L}\;\;\;and\;\;\;f=f(t),$$
where $g_{L}$ is an Einstein metric. Moreover, if $\lambda\geq0$ then $(L,g_{L})$ has non-negative Ricci curvature, and if it is Ricci flat, then $\psi$ is constant, i.e., $(M^{n},g)$ is a Riemannian product.
\end{theorem}

\section{Preliminaries}
\label{Preliminaries}

Throughout this section we recall some informations and basic results that will be useful in the proof of our main result. Firstly, by the trace of the fundamental equation (\ref{eqqEinstein}), we verify the relation
\begin{equation}\label{eqtrace}
R+\Delta f-\frac{1}{m}|\nabla f|^{2}=\lambda n.
\end{equation} For sake of simplicity, we now rewrite equation (\ref{eqqEinstein}) in the tensorial language as follows
\begin{equation}\label{eq:tensorial}
R_{ij}+\nabla_{i}\nabla_{j}f-\frac{1}{m}\nabla_{i}f\nabla_{j}f=\lambda g_{ij}.
\end{equation}

In order to proceed, we recall two special tensors in the study of curvature for a Riemannian manifold $(M^n,\,g),\,n\ge 3.$  The first one is the Weyl tensor $W$ which is defined by the following decomposition formula
\begin{eqnarray}
\label{weyl}
W_{ijkl}&=&R_{ijkl}-\frac{1}{n-2}\big(R_{ik}g_{jl}+R_{jl}g_{ik}-R_{il}g_{jk}-R_{jk}g_{il}\big) \nonumber\\
 &&+\frac{R}{(n-1)(n-2)}\big(g_{jl}g_{ik}-g_{il}g_{jk}\big),
\end{eqnarray}
where $R_{ijkl}$ stands for the Riemann curvature operator $Rm,$ whereas the second one is the Cotton tensor $C$ given by
\begin{equation}
\label{cotton} \displaystyle{C_{ijk}=\nabla_{i}R_{jk}-\nabla_{j}R_{ik}-\frac{1}{2(n-1)}\big(\nabla_{i}R
g_{jk}-\nabla_{j}R g_{ik}).}
\end{equation} Note that $C_{ijk}$ is skew-symmetric in the first two indices and trace-free in any two indices, i.e.,
$$C_{ijk}=-C_{jik}\;\;\;and\;\;\;g^{ij}C_{ijk}=g^{ik}C_{ijk}=0.$$
These two above tensors are related as follows
\begin{equation}\label{cottonwyel}
\displaystyle{C_{ijk}=-\frac{(n-2)}{(n-3)}\nabla_{l}W_{ijkl},}
\end{equation}provided $n\ge 4.$
Now, we recall a well-known tensor that was introduced by Bach \cite{bach} in the study of conformal relativity, namely, the Bach tensor. On a Riemannian manifold $(M^n,g)$, $n\geq 4,$ the Bach tensor is defined in term of the components of the Weyl tensor $W_{ikjl}$ as follows
\begin{equation}\label{bach}
B_{ij}=\frac{1}{n-3}\nabla_{k}\nabla_{l}W_{ikjl}+\frac{1}{n-2}R_{kl}W_{ikjl},
\end{equation}
while for $n=3$ it is given by
\begin{equation}
\label{bach3} B_{ij}=\nabla_{k}C_{kij}.
\end{equation} We say that $(M^n,g)$ is Bach-flat when $B_{ij}=0.$ It is easy to check that locally conformally flat metrics as well as Einstein metrics are Bach-flat. Moreover, for $4$-dimensional case, we have that, on any compact manifold $(M^{4},g)$, Bach-flat metrics are precisely the critical point of the {\it conformally invariant} functional on the space of the metrics,
$$\mathcal{W}(g)=\int_{M}|W_{g}|^{2}dV_{g},$$
for more details see, for example \cite{besse} or \cite{derd1}.
Furthermore, it is worth reporting here the following interesting formula for the divergence of the Bach tensor
\begin{equation}\label{divBACH}
\nabla_{p}B_{pi}=\frac{n-4}{(n-2)^{2}}C_{ijk}R_{jk}.
\end{equation}
We refer reader to \cite[Lemma 5.1]{CaoChen}, for its proof.

For the purposes of this work, let us recall some well-known properties about quasi-Einstein manifolds. For more detail, see \cite{CSW,CC,CPW,Ribeiro,Wang} and references therein. The next lemma, for instance, can be found in \cite{CSW}.

\begin{lemma}\label{L1}
Let $(M^{n},g,f,\lambda)$ be a quasi-Einstein manifold. Then we have:
$$\frac{1}{2}\nabla_{i}R+\frac{(n-1)}{m}\lambda\nabla_{i}f=\frac{m-1}{m}R_{is}\nabla_{s}f+\frac{1}{m}R\nabla_{i}f.$$
\end{lemma}

Now, we remembered the following 3-tensor defined in \cite{CC},
\begin{eqnarray}\label{TensorD}
D_{ijk}&=&\frac{1}{n-2}(R_{jk}\nabla_{i}f-R_{ik}\nabla_{j}f)+\frac{1}{(n-1)(n-2)}(g_{jk}R_{is}\nabla_{s}f-g_{ik}R_{js}\nabla_{s}f)\nonumber\\
&&-\frac{R}{(n-1)(n-2)}(g_{jk}\nabla_{i}f-g_{ik}\nabla_{j}f).
\end{eqnarray}
Note that $D_{ijk}$ has the same symmetry properties as the Cotton tensor. Still in \cite{CC}, the authors showed that $D_{ijk}$ is related to the Cotton tensor $C_{ijk}$ and the Weyl tensor $W_{ijkl}$.
Since its proof is very short, we include here for the sake of completeness. More precisely, we have the following identity.

\begin{lemma}(Chen-He, \cite{CC})\label{auxCotton}
Let $(M^{n},g,f)$ be a quasi-Einstein manifold. Then the following identity holds:
$$C_{ijk}=\frac{m+n-2}{m}D_{ijk}-W_{ijkl}\nabla_{l}f.$$
\end{lemma}
\begin{proof}
First of all, substitute (\ref{eq:tensorial}) into (\ref{cotton}) to deduce
\begin{eqnarray*}
C_{ijk}&=&(\nabla_{j}\nabla_{i}\nabla_{k}f-\nabla_{i}\nabla_{j}\nabla_{k}f)+\frac{1}{m}(\nabla_{j}f\nabla_{i}\nabla_{k}f-\nabla_{i}f\nabla_{j}\nabla_{k}f)\\
&&-\frac{1}{2(n-1)}(\nabla_{i}R g_{jk}-\nabla_{j}Rg_{ik})\\
&=&-R_{ijkl}\nabla_{l}f+\frac{1}{m}(\nabla_{j}f\nabla_{i}\nabla_{k}f-\nabla_{i}f\nabla_{j}\nabla_{k}f)\\
&&-\frac{1}{2(n-1)}(\nabla_{i}R g_{jk}-\nabla_{j}Rg_{ik}),
\end{eqnarray*}
where in the last step we use the Ricci identities.

Next, by Eq. (\ref{eq:tensorial}) again and Lemma~\ref{L1}, we have that
\begin{eqnarray*}
C_{ijk}&=&-R_{ijkl}\nabla_{l}f+\frac{1}{m}(R_{jk}\nabla_{i}f-R_{ik}\nabla_{j}f)\\
&&-\frac{\lambda}{m}(\nabla_{i}fg_{jk}-\nabla_{j}fg_{ik})-\frac{1}{2(n-1)}(\nabla_{i}R g_{jk}-\nabla_{j}Rg_{ik})\\
&=&-R_{ijkl}\nabla_{l}f+\frac{1}{m}(R_{jk}\nabla_{i}f-R_{ik}\nabla_{j}f)\\
&&-\frac{m-1}{m(n-1)}(R_{il}\nabla_{l}fg_{jk}-R_{jl}\nabla_{l}fg_{ik})-\frac{R}{m(n-1)}(\nabla_{i}f g_{jk}-\nabla_{j}fg_{ik}).
\end{eqnarray*}
Finally, substituting (\ref{weyl}) in the above expression, and after some computation, we get
\begin{eqnarray*}
C_{ijk}&=&-W_{ijkl}\nabla_{l}f+D_{ijk}+\frac{1}{m}(R_{jk}\nabla_{i}f-R_{ik}\nabla_{j}f)\\
&&+\frac{1}{m(n-1)}(R_{il}\nabla_{l}fg_{jk}-R_{jl}\nabla_{l}fg_{ik})-\frac{R}{m(n-1)}(\nabla_{i}f g_{jk}-\nabla_{j}fg_{ik})\\
&=&-W_{ijkl}\nabla_{l}f+D_{ijk}+\frac{n-2}{m}D_{ijk}\\
&=&-W_{ijkl}\nabla_{l}f+\frac{m+n-2}{m}D_{ijk},
\end{eqnarray*}
as desired.
\end{proof}

Under these notations we get the following lemma.

\begin{lemma}\label{L0}
Let $(M^{n},g,f,\lambda),$ $n\geq4,$ be a quasi-Einstein manifold. Then we have:
$$B_{ij}=\frac{m+n-2}{n(n-2)}\nabla_{k}D_{kij}+\frac{n-3}{(n-2)^{2}}C_{kji}\nabla_{k}f+\frac{1}{m(n-2)}W_{ikjl}\nabla_{k}f\nabla_{l}f.$$
\end{lemma}
\begin{proof}
To beginning with, we use (\ref{eq:tensorial}) and (\ref{cottonwyel}) to infer
\begin{eqnarray*}
\nabla_{k}(W_{ikjl}\nabla_{l}f)&=&\nabla_{k}W_{ikjl}\nabla_{l}f+W_{ikjl}\left(-R_{kl}+\frac{1}{m}\nabla_{k}\nabla_{l}f+\lambda g_{kl}\right)\\
&=&\frac{n-3}{n-2}C_{lji}\nabla_{l}f-W_{ikjl}R_{kl}+\frac{1}{m}W_{ikjl}\nabla_{k}f\nabla_{l}f.
\end{eqnarray*}

Then, from Lemma~\ref{auxCotton}, we immediately obtain
\begin{eqnarray}\label{auxB}
\nabla_{k}C_{kij}+W_{ikjl}R_{kl}&=&\frac{m+n-2}{m}\nabla_{k}D_{kij}+\frac{n-3}{n-2}C_{lji}\nabla_{l}f\nonumber\\
&&+\frac{1}{m}W_{ikjl}\nabla_{k}f\nabla_{l}f.
\end{eqnarray}

Therefore, substituting (\ref{cottonwyel}) in (\ref{bach}) we get
\begin{equation}\label{bachAux}
(n-2)B_{ij}=\nabla_{k}C_{kij}+W_{ikjl}R_{kl},
\end{equation}
which combined with (\ref{auxB}) gives the requested result.
\end{proof}

Now, restricting to the steady quasi-Einstein case, we can take $u=e^{-\frac{f}{m}}$ to deduce, after a straightforward computation, the following identities
$$\nabla_{i}u=-\frac{u}{m}\nabla_{i} f$$
and
$$\nabla_{i}\nabla_{j}u=-\frac{u}{m}\left(\nabla_{i}\nabla_{j}f-\frac{1}{m}\nabla_{i}f\nabla_{j}f\right)=\frac{u}{m}R_{ij}.$$
Hence, taking the trace in the last equality, it is easy to verify that
\begin{equation}\label{deltau}
\Delta u=\frac{1}{u}|\nabla u|^{2}-\frac{u}{m}\Delta f=\frac{u}{m}R.
\end{equation}
Moreover, substituting (\ref{mu}) into (\ref{deltau}), we get
\begin{equation}\label{auxnormau}
m(m-1)|\nabla u|^{2}+Ru^{2}=m\mu.
\end{equation}

We finalize this section with a lemma that will be very useful for our purposes, whose its proof can be check in \cite{Ribeiro}.

\begin{lemma}\label{estimateu}
Let $(M^{n},g,f,m>1)$ be a complete (noncompact) steady quasi Einstein manifold with positive Ricci curvature and such that $f$ has at last one critical point. Then there exist positive constant $c_{1}$ and $c_{2}$ such that the function $u=e^{-\frac{f}{m}}$ satisfies the estimates
$$c_{1}r(x)-c_{2}\leq u(x)\leq \sqrt{\frac{\mu}{m-1}}r(x)+|u(p)|,$$
where $r(x)=d(p,x)$ is the distance function from some fixed critical point $p\in M,$ $c_{1}$ and $c_{2}$ are positive constants depending only on $n$ and the geometry of $g_{ij}$ on the unit ball $B_{p}(1).$
\end{lemma}

\section{ Quasi-Einstein manifold with ${\rm div}^{4}W=0$}

In this section we shall prove Theorems \ref{thmA} and \ref{thmB} announced in Section \ref{intro}. To do this, under our assumption, we shall first derive a useful integral formula for the norm square of the Cotton tensor for steady quasi-Einstein metric with zero radial Weyl tensor. This formula plays an important role in our conclusion of the desired theorems.

\subsection{Proof of the Theorem~\ref{thmA}}
\begin{proof}
Firstly, by direct computation using (\ref{bachAux}), we get
\begin{eqnarray*}
(n-2){\rm div}^{2}B&=&\nabla_{j}(\nabla_{i}\nabla_{k}C_{kij}+\nabla_{i}W_{ikjl}R_{kl}+W_{ijkl}\nabla_{i}R_{kl})\\
&=&\nabla_{j}(\nabla_{i}\nabla_{k}C_{kij}+\nabla_{i}W_{ikjl}R_{kl}+\frac{1}{2}W_{ijkl}(\nabla_{i}R_{kl}-\nabla_{k}R_{il}).
\end{eqnarray*}
In the sequel, using the expressions (\ref{cotton}) and (\ref{cottonwyel}) and the fact that the Weyl tensor $W$ has null trace, we arrive at
\begin{eqnarray}\label{div2B}
(n-2){\rm div}^{2}B&=&{\rm div}^{3}C+\nabla_{j}\left(-\frac{n-3}{n-2}C_{ljk}R_{kl}+\frac{1}{2}W_{ikjl}C_{ikl}\right).
\end{eqnarray}

Next, consider a critical point $p\in M$ and take the ball $B_{p}(s)$ of radius s centered at $p.$ Also, let $\phi:M\rightarrow\mathbb{R}$ be a smooth test function defined on $M.$ Thus, integrating (\ref{div2B}) by parts and changing the indices conveniently, we obtain
\begin{eqnarray*}
(n-2)\int_{B_{p}(s)}\phi {\rm div}^{2}BdV_{g}&=&-\frac{n-3}{n-2}\int_{B_{p}(s)}C_{ijk}\nabla_{i}\phi R_{jk}dV_{g}+\frac{n-3}{n-2}\int_{\partial B_{p}(s)}\phi C_{ijk}R_{jk}\nu_{i}d\sigma\\
&&+\frac{1}{2}\int_{B_{p}(s)}W_{ijkl}\nabla_{l}\phi C_{ijk}dV_{g}-\frac{1}{2}\int_{\partial B_{p}(s)}\phi W_{ijkl}C_{ijk}\nu_{l}d\sigma\\
&&+\int_{B_{p}(s)}\phi {\rm div}^{3}CdV_{g},
\end{eqnarray*}
where $\nu$ is the outward unit normal on $\partial B_{p}(s)$ and $d\sigma$ is its volume form.

Proceeding, by Lemma~\ref{auxCotton} together with fact which our manifold satisfies zero radial Weyl curvature (i.e., $W_{ijkl}\nabla_{l}f=0$), we get
\begin{eqnarray*}
(n-2)\int_{B_{p}(s)}\phi {\rm div}^{2}BdV_{g}&=&-\frac{n-3}{n-2}\int_{B_{p}(s)}C_{ijk}\nabla_{i}\phi R_{jk}dV_{g}+\frac{n-3}{n-2}\int_{\partial B_{p}(s)}\phi C_{ijk}R_{jk}\nu_{i}d\sigma\\
&&+\frac{m+n-2}{2m}\Big\{\int_{B_{p}(s)}W_{ijkl}\nabla_{l}\phi D_{ijk}dV_{g}\\
&&-\int_{\partial B_{p}(s)}\phi W_{ijkl}D_{ijk}\nu_{l}d\sigma\Big\}+\int_{B_{p}(s)}\phi {\rm div}^{3}CdV_{g}.
\end{eqnarray*}
So, as the tensor $D$ has the same skew-symmetric of the Cotton tensor, we arrive at
\begin{eqnarray}\label{EqAdiv2aux}
(n-2)\int_{B_{p}(s)}\phi {\rm div}^{2}BdV_{g}&=&-\frac{n-3}{n-2}\int_{B_{p}(s)}C_{ijk}\nabla_{i}\phi R_{jk}dV_{g}+\frac{n-3}{n-2}\int_{\partial B_{p}(s)}\phi C_{ijk}R_{jk}\nu_{i}d\sigma\nonumber\\
&&+\frac{m+n-2}{m(n-2)}\Big\{\int_{B_{p}(s)}W_{ijkl}\nabla_{l}\phi R_{jk}\nabla_{i}fdV_{g}\nonumber\\
&&-\int_{\partial B_{p}(s)}\phi W_{ijkl}R_{jk}\nabla_{i}f\nu_{l}d\sigma\Big\}+\int_{B_{p}(s)}\phi {\rm div}^{3}CdV_{g}\nonumber\\
&=&-\frac{n-3}{n-2}\int_{B_{p}(s)}C_{ijk}\nabla_{i}\phi R_{jk}dV_{g}+\frac{n-3}{n-2}\int_{\partial B_{p}(s)}\phi C_{ijk}R_{jk}\nu_{i}d\sigma\nonumber\\
&&+\int_{B_{p}(s)}\phi {\rm div}^{3}CdV_{g},
\end{eqnarray}
where in the last step we use the condition $W_{ijkl}\nabla_{l}f=0$ again.

On the other hand, multiplying the equation (\ref{divBACH}) by $\phi$ and integrating by parts, we deduce
\begin{eqnarray}\label{EqBdiv2aux}
(n-2)\int_{B_{p}(s)}\phi {\rm div}^{2}BdV_{g}&=&\frac{n-4}{n-2}\int_{B_{p}(s)}\phi\nabla_{i}(C_{ijk}R_{jk})dV_{g}\nonumber\\
&=&\frac{n-4}{n-2}\left\{\int_{B_{p}(s)}\nabla_{i}(\phi C_{ijk}R_{jk})dV_{g}-\int_{B_{p}(s)}\nabla_{i}\phi C_{ijk}R_{jk}dV_{g}\right\}\nonumber\\
&=&\frac{n-4}{n-2}\int_{\partial B_{p}(s)}\phi C_{ijk}R_{jk}\nu_{i}d\sigma-\frac{n-4}{n-2}\int_{B_{p}(s)} C_{ijk}\nabla_{i}\phi R_{jk}dV_{g}.
\end{eqnarray}
Thus, comparing (\ref{EqAdiv2aux}) with (\ref{EqBdiv2aux}) and using the Lemma~\ref{auxCotton}, we have
\begin{eqnarray}\label{intCaux}
\int_{B_{p}(s)}\nabla_{i}\phi C_{ijk}R_{jk}dV_{g}&=&\int_{\partial B_{p}(s)}\phi C_{ijk}R_{jk}\nu_{i}d\sigma+(n-2)\int_{B_{p}(s)}\phi {\rm div}^{3}CdV_{g}\nonumber\\
&=&\frac{m+n-2}{m}\int_{\partial B_{p}(s)}\phi D_{ijk}R_{jk}\nu_{i}d\sigma+(n-2)\int_{B_{p}(s)}\phi {\rm div}^{3}CdV_{g}.
\end{eqnarray}
Now, from the definition of the tensor $D_{ijk},$ it is easy to check that
\begin{eqnarray*}
D_{ijk}R_{jk}\nu_{i}&=&\frac{1}{n-2}|Ric|^{2}\langle \nabla f,\nu\rangle-\frac{n}{(n-1)(n-2)}R_{ik}R_{jk}\nabla_{i}f\nu_{j}\\
&&+\frac{2R}{(n-1)(n-2)}R_{ij}\nabla_{i}f\nu_{j}-\frac{R^{2}}{(n-1)(n-2)}\langle\nabla f,\nu\rangle,
\end{eqnarray*}
which joint with fact that $|R_{ij}|\leq R$ (this follows directly from our hypothesis which $M$ has positive Ricci curvature), allows us to conclude the following estimate
\begin{equation}\label{estD}
|D_{ijk}R_{jk}\nu_{i}|\leq CR^{2}|\nabla f|,
\end{equation}
for some constant $C>0.$ Furthermore, since we are working in the steady case, the scalar curvature is nonnegative (cf. \cite{Wang} for more details) and consequently, from (\ref{auxnormau}) we get
$$|\nabla u|^{2}\leq\frac{\mu}{m-1}\;\;\;\;and\;\;\;\;u^{2}R\leq m\mu.$$
Hence, (\ref{estD}) becomes
$$|D_{ijk}R_{jk}\nu_{i}|\leq C\frac{1}{u^{5}},$$
where we denote the same constant for simplicity.

Next, consider $s$ sufficiently large such that $c_{1}s-c_{2}$ is a positive number, where $c_{1}$ and $c_{2}$ are constant provided in the Lemma~\ref{estimateu}. Thus, by the inequality obtained in Lemma~\ref{estimateu} and taking our test function as $\phi=u^{-n+\frac{11}{2}},$ we deduce
\begin{eqnarray}\label{estBG}
\left|\int_{\partial B_{p}(s)}\phi D_{ijk}R_{jk}\nu_{i}d\sigma\right|&\leq&\int_{\partial B_{p}(s)}u^{-n+\frac{11}{2}} |D_{ijk}R_{jk}\nu_{i}|d\sigma\nonumber\\
&\leq& C\int_{\partial B_{p}(s)}\frac{1}{u^{n-\frac{1}{2}}}d\sigma\nonumber\\
&\leq&\frac{C}{(c_{1}s-c_{2})^{n-\frac{1}{2}}}Area(\partial B_{p}(s)).
\end{eqnarray}
Therefore, since we are assuming positive Ricci curvature, it follows from the well-known Bishop-Gromov's theorem, that
$$Area(\partial B_{p}(s))\leq\widetilde{C}s^{n-1},$$
where $\widetilde{C}$ is a positive constant.
Thus, returning to inequality (\ref{estBG}), we deduce
\begin{eqnarray*}
\left|\int_{\partial B_{p}(s)}\phi D_{ijk}R_{jk}\nu_{i}d\sigma\right|&\leq&C\left(\frac{s}{c_{1}s-c_{2}}\right)^{n-\frac{1}{2}}\frac{1}{\sqrt{s}},
\end{eqnarray*}
again we consider the same constant. In particular, if we take $s\rightarrow\infty$ in (\ref{intCaux}), then it is easy to verify that
\begin{eqnarray*}
\int_{M}u^{-n+\frac{11}{2}} {\rm div}^{3}CdV_{g}&=&\frac{2n-11}{2m(n-2)}\int_{M}u^{-n+\frac{11}{2}} C_{ijk}\nabla_{i}fR_{jk}dV_{g}\\
&=&\frac{2n-11}{4m(n-2)}\int_{M}u^{-n+\frac{11}{2}} C_{ijk}(\nabla_{i}fR_{jk}-\nabla_{j}fR_{ik})dV_{g}\\
&=&\frac{2n-11}{4m}\int_{M}u^{-n+\frac{11}{2}} C_{ijk}D_{ijk}dV_{g},
\end{eqnarray*}
where we use the skew-symmetry of Cotton tensor jointly with the definition of the auxiliary tensor $D$. So, we may use Lemma~\ref{auxCotton} in order to get the following integral formula for steady quasi-Einstein metric with zero radial Weyl tensor,
\begin{eqnarray}\label{intID}
\int_{M}u^{-n+\frac{11}{2}} {\rm div}^{3}CdV_{g}=\frac{(2n-11)}{4(m+n-2)}\int_{M}u^{-n+\frac{11}{2}} |C_{ijk}|^{2}dV_{g}.
\end{eqnarray}

Finally, since we are suppose that $M$ has fourth order divergence-free Weyl tensor (i.e., ${\rm div}^{4}W=0$), it follows from (\ref{intID}) that $M$ has null Cotton tensor, which is equivalent to say, using Eq. (\ref{cottonwyel}), that $M$ has harmonic Weyl tensor. This is what we wanted to prove.
\end{proof}

\subsection{Conclusion of the proof of Corollary~\ref{corA} and Theorem~\ref{thmB}}
\begin{proof}
Under the conditions of Theorem 2, since we already know that $M$ has null Cotton tensor, then by Lemma~\ref{L0} we get immediately that $M$ is a Bach-flat manifold and consequently, the Corollary~\ref{corA} and Theorem~\ref{thmB} follows from Theorem 2 and Corollary 1 in \cite{Ribeiro}, respectively.
\end{proof}

\end{document}